\documentclass[11pt]{article}

\usepackage[english]{babel}
\usepackage[ansinew]{inputenc}
\usepackage{verbatim}
\usepackage{amsmath, amsthm, amsfonts, amssymb}
\usepackage{latexsym}
\usepackage{bbm}
\usepackage{mathrsfs}
\usepackage{tikz}
\usepackage{graphicx}
\usepackage{arydshln}
\usepackage{multirow}
\usepackage{subfig}

\usetikzlibrary{matrix,arrows}

\tikzset{
  >=latex,
  inner sep=0pt,
  outer sep=2pt,
  mark coordinate/.style={inner sep=0pt,outer sep=0pt,minimum size=3pt,fill=black,circle}
}

\newtheorem{thm}{Theorem}[section]
\newtheorem{cor}[thm]{Corollary}

\newtheorem{lemma}[thm]{Lemma}

\newtheorem{proposition}[thm]{Proposition}

\theoremstyle{definition}
\newtheorem{definition}[thm]{Definition}

\newtheorem{remark}[thm]{Remark}

\theoremstyle{remark}

\newcommand{\IR}{\mathbb{R}}

\newcommand{\mcL}{\mathcal{L}}
\newcommand{\mcF}{\mathcal{F}}

\newcommand{\CC}{\mathscr{C}}

\newcommand{\Ren}{{\textup{\tiny R}}}
\newcommand{\Pri}{{\textup{\tiny P}}}
\newcommand{\Dua}{{\textup{\tiny D}}}
\newcommand{\ol}[1]{\overline{#1}}
\newcommand{\CGCC}{\CC_{\textup{\tiny GCC}}}
\newcommand{\veps}{\varepsilon}

\DeclareMathOperator{\im}{im}
\DeclareMathOperator{\Gr}{Gr}
\DeclareMathOperator{\inter}{int}
\DeclareMathOperator{\tr}{tr}
\DeclareMathOperator{\id}{id}

\def\R{\IR}
\def\CG{\CC}

\def\CR{\mathcal{R}}
\def\PG{\mathcal{P}}
\def\DG{\mathcal{D}}
\def\PGm{\PG_m(C)}
\def\DGm{\DG_m(C)}
\def\SG{\Sigma}
\def\SGm{\Sigma_m(C)}
\def\FPR{\mcF^\Pri_\Ren}
\def\FDR{\mcF^\Dua_\Ren}
\def\SR{\Sigma_\Ren}
\def\sa{\sphericalangle}
\def\inte{\mathrm{int}}
\def\GL{\mathrm{GL}}
\def\GR{\Gr_{n,m}}
\def\a{\alpha}
\def\Es{\mathcal{E}}
\def\Ds{\mathcal{D}}
\def\RD{\R^{m\times n}_{-}}
\def\s{\sigma}
\def\Ex{\mathbb{E}}
\def\capp{\mathsf{cap}}

\usepackage[pdftex, pdfpagelabels]{hyperref}
\hypersetup{
  plainpages=false,
  colorlinks,
  citecolor=black,
  filecolor=black,
  linkcolor=black,
  urlcolor=black,
  pdfpagemode = UseNone,
}

\begin{document}

\title{\bf A coordinate-free condition number for convex programming}
\author{Dennis Amelunxen$^\ast$ and Peter B\"urgisser\thanks{Institute of Mathematics, University of Paderborn, Germany. 
Partially supported by DFG grant BU 1371/2-1 and DFG Research Training Group on Scientific Computation GRK 693 (PaSCo GK).}
        \\ University of Paderborn
        \\ \{damelunx,pbuerg\}@math.upb.de
        }
\date{June 5, 2012}
\maketitle 

\begin{abstract}
We introduce and analyze a natural geometric version of Renegar's condition number~$\CR$
for the homogeneous convex feasibility problem 
associated with a regular cone~$C\subseteq\R^n$.
Let $\Gr_{n,m}$ denote the Grassmann manifold of $m$-dimensional
linear subspaces of $\R^n$ and consider 
the projection distance 
$d_p(W_1,W_2) := \| \Pi_{W_1} - \Pi_{W_2}\|$ (spectral norm) 
between $W_1,W_2\in \Gr_{n,m}$,
where $\Pi_{W_i}$ denotes the orthogonal projection onto $W_i$.
We call 
$\CG(W) := \max\{ d_p(W,W')^{-1} \mid W' \in \Sigma_m\}$
the {\em Grassmann condition number} of~$W\in\GR$,
where the set of ill-posed instances $\Sigma_m\subset\Gr_{n,m}$ is defined 
as the set of linear subspaces touching~$C$. 
We show that if $W =\im(A^T)$ for a matrix $A\in\R^{m\times n}$, then 
$\CG(W) \le \CR(A) \le \CG(W)\, \kappa(A)$,
where $\kappa(A) =\|A\| \|A^\dagger\|$ denotes the 
matrix condition number.
This extends work by Belloni and Freund in
Math.~Program.~119:95--107 (2009). 
Furthermore,
we show that $\CG(W)$ can as well be characterized 
in terms of the Riemannian distance metric on $\GR$. 
This differential geometric characterization of $\CG(W)$
is the starting point of the sequel [arXiv:1112.2603] to this paper, 
where the first probabilistic analysis of Renegar's condition number 
for an arbitrary regular cone~$C$ is achieved. 
\end{abstract}

\smallskip


\section{Introduction}

It is by now a well established fact~\cite{rene:95b,rene:95a,vaye:95,frve:99,epfr:00,chcu:04,chcu:06} 
that the running time of a variety of algorithms in linear programming 
can be efficiently bounded in terms of a notion of condition. 
The condition is defined as a measure of sensitivity of the output 
with respect to small perturbations of the input.  
Different variants of this notion exists: 
the most common is the one originally introduced by 
Jim Renegar~\cite{rene:94,rene:95b,rene:95a}.

The analysis of the probability distribution of the condition of random input data 
is a thoroughly studied subject
for a variety of numerical problems,
 compare the recent survey~\cite{buergICM} for references. 
It has recently received increased attention through the concept of {\em smoothed analysis}, 
introduced by Spielman and Teng~\cite{ST:04}, who also
managed to perform a smoothed analysis of Renegar's condition number for \emph{linear programming}~\cite{DST:09}. 

The motivation of the present work is to extend such probabilistic analyses to \emph{general convex programming}, notably to \emph{second-order} and \emph{semidefinite programming}.
Renegar's condition number is hard to analyze directly. 
In fact, behind the analysis in~\cite{DST:09} there is an intermediate concept, 
the so-called GCC-condition number~\cite{goff:80,ChC:01} tailored to the LP cone $\R^n_+$, 
that has nice geometric characterizations that 
facilitate its probabilistic analysis, see \cite{BCL:08a,AB:08}. 
All known probabilistic analyses of condition numbers for linear programming heavily rely on 
the product structure of the cone $\R_+\times\cdots\times \R_+$ and thus cannot be extended to general convex cones. 
In this paper we introduce a coordinate-free, geometric notion of condition of independent interest, 
that allows to overcome this difficulty at the price of working in the intrinsic geometric setting of Grassmann manifolds.

\subsection{Renegar's condition number}\label{se:Rene}

A {\em regular cone} $C\subset\IR^n$ is 
a closed convex cone with nonempty interior that 
does not contain a nontrivial linear subspace. 
The \emph{dual cone}\footnote{Some
authors call $\breve{C}$ the polar cone and $-\breve{C}$ the dual cone.}
 of $C$ is defined as 
$\breve{C} := \{z\in\IR^n\mid \forall x\in C : z^Tx\leq 0\}$.
If $C$ is regular, then $\breve{C}$ is regular as well.
We call $C$ \emph{self-dual} if $\breve{C}=-C$. 
Important cones for applications are, 
besides the LP case $C=\R_+^n$, 
the second order cones $C=\mcL^{n_1}\times\ldots\times \mcL^{n_k}$, 
where $\mcL^n:=\{x\in\IR^n\mid x_n\geq (x_1^2 + \cdots + x_{n-1}^2)^{1/2}\}$, 
and the cone of positive semidefinite matrices.
All these cones are self-dual.

In the following we fix a regular cone $C\subset\IR^n$.
Throughout the paper we assume that $1\leq m <n$.
The \emph{homogeneous convex feasibility problem} is to decide 
for a given matrix $A\in\IR^{m\times n}$ the alternative 
\begin{align}
   &\exists x\in\IR^n\setminus \{0\} \;  \text{ s.t.} \quad A x=0 \,,\; x\in \breve{C} \; , \label{eq:P}\tag{P}\\
   &\exists y\in\IR^m\setminus  \{0\} \;  \text{ s.t. } \quad A^Ty\in C \; . \label{eq:D}\tag{D}
\end{align}
We define the sets of primal and dual feasible instances with respect to~$C$, respectively, by 
\begin{align}\label{eq:star1}
   \FPR & := \big\{A\in\IR^{m\times n} \mid \eqref{eq:P} \text{ is feasible} \big\}
       = \{A\mid \ker(A)\cap \breve{C}\neq \{0\} \} \; , \\[2mm] \label{eq:star2}
  \FDR & := \big\{A\in\IR^{m\times n} \mid \eqref{eq:D} \text{ is feasible} \big\}
 = \RD \; \cup\; \{A\mid \im(A^T)\cap C\neq \{0\} \} \; ,
\end{align}
where $\RD$ denotes the set of rank-deficient matrices in $\IR^{m\times n}$. 
Here are the most relevant properties of the sets $\FPR$ and $\FDR$:
they are closed, 
both their boundaries coincide with the \emph{set of ill-posed inputs} $\SR := \FPR\cap\FDR$, 
and $\IR^{m\times n} = \FPR\cup\FDR$.
(We shall state and prove related statements in Section~\ref{se:GCN}).
We also note that $\FPR$ and $\FDR$ are invariant under 
the action of the general linear group $\GL(m)$ on $\R^{m\times n}$ by left multiplication. 

\emph{Renegar's condition}~\cite{rene:94,rene:95b,rene:95a} 
is defined as the function
\begin{equation}\label{def:RCN}
 \CR:=\CR_C\colon\IR^{m\times n}\setminus \{0\}\to[1,\infty] \;,\quad \CR_C(A) := \frac{\|A\|}{d(A,\SR)} \; , 
\end{equation}
where $\|A\|$ denotes the spectral norm, and $d(A,\SR)=\min\{\|A-A'\|\mid A'\in\SR\}$. 
One can also characterize $\CR(A)^{-1}$ as follows:
\[
 \CR(A)^{-1} = \max\left\{r\left| \|\Delta A\|\leq r\cdot \|A\| \Rightarrow 
 \left(\begin{matrix} A+\Delta A\in\FPR & \text{if $A\in\FPR$} \\ 
  A+\Delta A\in\FDR & \text{if $A\in\FDR$}\end{matrix}\right)\right.\right\} \; . 
\]
That is, $\CR(A)^{-1}$ is the maximum $r\ge 0$ such that all
pertubations $\Delta A$ of norm at most $r\|A\|$ do not change the
feasibility status of $A$.

\subsection{The Grassmann condition number}\label{se:GCN}

The \emph{Grassmann manifold} $\Gr_{n,m}$ is defined as the set of 
$m$-dimensional linear subspaces~$W$ of~$\R^n$. 
It is well known that $\Gr_{n,m}$  is a compact smooth 
manifold on which the orthogonal group $O(n)$
acts transitively, see for instance~\cite{booth}.

Following~\eqref{eq:star1} and \eqref{eq:star2}, we define 
the sets of $m$-dimensional {\em primal feasible subspaces}, 
and {\em dual feasible subspaces} with respect to the regular cone~$C$, respectively, by
\begin{equation*}\label{eq:def-F^P_G,F^D_G}
  \PG_m(C)  := \big\{W\in\Gr_{n,m} \mid W^\bot\cap \breve{C}  \neq \{0\} \big\} \;,\quad 
  \DG_m(C) := \big\{W\in\Gr_{n,m} \mid  W\cap C  \neq \{0\} \big\} .
\end{equation*}
Note that, unlike in~\eqref{eq:P} and \eqref{eq:D}, there is no structural difference between primal and dual feasibility.
The primal feasibility of $W$ with respect to $C$ just means the dual feasibility of $W^\bot$ with 
respect to $\breve{C}$. 
In terms of the involution\footnote{In this context we interpret $\iota_m$ as an involution, since $\iota_{n-m}\circ \iota_m=\id_m$ and $\iota_m\circ \iota_{n-m}=\id_{n-m}$, where $\id_m$ denotes the identity map on $\Gr_{n,m}$.}
$$
 \iota_m\colon\Gr_{n,m} \to \Gr_{n,n-m},\; W \mapsto W^\bot ,
$$
this can be expressed as 
$\iota_m\big(\PG_m(C)\big) = \DG_{n-m}(\breve{C})$.

We claim that $\Gr_{n,m}=\PG_m(C)\cup\DG_m(C)$.
For this recall the well-known 
theorem on alternatives\footnote{Although this theorem 
of alternatives is folklore, we could not
find a perfect reference for it in the literature.
See for example~\cite[Thm.~3]{BB:73} for a complex version of~\eqref{eq:alt}.
The given proof is easily adapted to the real case.}, 
which for $C=\IR_+^n$ is also known as Farkas' Lemma~\cite{farkas:01}:
for $W \in \Gr_{n,m}$ we have 
\begin{equation}\label{eq:alt}
W\cap\inte(C)\ne\emptyset \iff W^\bot\cap\breve{C}= \{0\}.
\end{equation}
Now $W\not\in\PG_m(C)$ means 
$W^\bot\cap \breve{C}=\{0\}$, which by~\eqref{eq:alt} is equivalent to 
$W\cap \inte(C) \ne\emptyset$. 
This in particular implies $W\in\DG_m(C)$. 

By a similar reasoning we obtain the following characterization 
of the set of $m$-dimensional \emph{ill-posed subspaces} 
$\SG_m(C):=\PG_m(C)\cap\DG_m(C)$ with respect to~$C$:
\begin{equation}\label{eq:charS}
   \SG_m(C) = \{W\in\Gr_{n,m}\mid W\cap C \neq \{0\} \text{ and } W\cap \inte(C) = \emptyset \} \; .
\end{equation}
Thus $\SG_m(C)$ consists of the subspaces $W\in\Gr_{n,m}$ touching the cone~$C$.
As for the involution~$\iota_m$, we obtain
the following duality relations: 
\begin{equation}\label{eq:sym}
\iota_m(\PG_m(C)) = \DG_{n-m}(\breve{C}), \quad 
\iota_m(\DG_m(C)) = \PG_{n-m}(\breve{C}), \quad 
\iota_m(\SG_m(C)) = \SG_{n-m}(\breve{C}). 
\end{equation}

The {\em projection distance} $d_p(W_1,W_2)$ of $W_1,W_2\in \Gr_{n,m}$
is defined as the spectral norm 
$d_p(W_1,W_2) :=\|\Pi_{W_1}-\Pi_{W_2}\|$,
where $\Pi_{W_i}$ denotes the orthogonal projection onto~$W_i$, cf.~\cite[\S2.6]{GoLoan}. 
Clearly, this defines a metric and a corresponding topology on $\Gr_{n,m}$.
We will see, cf.~\eqref{eq:perp}, that
$d_p(W_1^\perp,W_2^\perp) = d_p(W_1,W_2)$ so that the 
involution $\iota_m$
preserves the projection distance.

The proof of the following basic topological result is provided at the end of the paper. 

\begin{proposition}\label{pro:Gtop}
\begin{enumerate}
\item[(1)] The sets $\PG_m(C)$ and $\DG_m(C)$ are both closed subsets of $\Gr_{n,m}$.
\item[(2)] The boundaries of $\PG_m(C)$ and $\DG_m(C)$ both coincide with $\SG_m(C)$. 
\end{enumerate}
\end{proposition}

We define now the Grassmann condition of $W\in \Gr_{n,m}$ as the inverse distance of $W$ 
to the set $\SG_m(C)$ of ill-posed subspaces, measured by the projection distance.

\begin{definition}\label{def:GCN}
The \emph{Grassmann condition} with respect to the regular cone~$C\subseteq\R^n$ 
is defined as the function
\[ 
 \CG_C\colon \GR \to [1,\infty] \;,\quad \CG_C(W) := \frac{1}{d_p(W,\SG_m(C))}  \; ,
\]
where $d_p(W,\SG_m(C)):=\min\{ d_p(W,W') \mid W' \in\SG_m(C)\}$.
\end{definition}

Because of \eqref{eq:sym} and since the involution $\iota_m$ preserves the 
projection distance, we have 
\begin{equation}\label{eq:SGdual}
\CG_C(W) = \CG_{\breve{C}}(W^\bot) \; .
\end{equation}
When the reference cone~$C$ is clear from the context
we simply write $\CG=\CC_C$ and similarly for $\PG_m$, $\DG_m$, and $\SG_m$.

\subsection{Main results} \label{se:main}

We first relate the Grassmann condition number to Renegar's condition number. 
Again we fix a regular cone $C\subseteq\R^n$. 

In the following let $\R^{m\times n}_*$ denote the set of $m\times n$-matrices of full rank~$m$. 
A subspace $W\in\GR$ can be represented by a matrix $A\in \IR^{m\times n}_*$ in the form 
$W = \im(A^T) = \ker(A)^\bot$. 
The numerical quality of the matrix $A$ is measured by its condition number 
$\kappa(A) =\|A\| \|A^\dagger\|$, where $A^\dagger\in\R^{n\times m}$ stands for the 
Moore-Penrose pseudoinverse of~$A$. 

We call a matrix $B\in\R^{m\times n}$ {\em balanced} iff $B B^T=I_m$.
It is easy to see that $\kappa(B)=1$ iff $B$ is balanced (cf.\ Lemma~\ref{le:bal}). 
The next result states that $\CG(W)$ equals Renegar's condition number $\CR(B)$ 
for a balanced representation~$B$ of the subspace~$W$. 

\begin{thm}\label{th:main}
For $B\in\IR^{m\times n}$ balanced and $W=\im(B^T)$
we have $\CG(W) =\CR(B)$. 
\end{thm}

Now we address the question to what extent $\CR(A)$ deviates from $\CG(W)$,
when we represent the subspace~$W$ by a nonbalanced matrix~$A$ such that 
$W=\im(A^T)$. 
As one may expect, this is quantified by the matrix condition number~$\kappa(A)$. 

\begin{thm}\label{th:BF}
For $A\in\IR^{m\times n}_*$ and $W=\im(A^T)$ we have
$$
  \CG(W) \;\leq\; \CR(A) \;\leq\; \kappa(A)\cdot \CG(W) \; .
$$
\end{thm}

By Theorem~\ref{th:main}, the left-hand inequality states that 
$\CR(B) \le \CR(A)$ where $B$ is balanced and $\im(B^T)=\im(A^T)$.
The right-hand inequality 
expresses the fact that a large condition $\CR(A)$ is either caused by a large $\CG(W)$, 
i.e., $W$ meeting/missing $C$ at small angle,
or caused by a large $\kappa(A)$, i.e., a badly conditioned matrix $A$ representing the subspace~$W$. 

\begin{remark}\label{re:CBA}
As a preconditioning process, we may replace the input matrix~$A$ by a balanced matrix~$B$ such 
that $\im(A^T)=\im(B^T)$. 
For instance, this may be achieved by a Gram-Schmidt orthogonalization 
of the columns of $A$. Then, for the three families of cones~$C$ 
corresponding to linear, second-order or semidefinite programming, 
we may apply the primal-dual interior-point method of~\cite{VRP:07} to decide the 
alternative (P)/(D) on input~$B$ with a number of interior-point iterations bounded by
$O\big(\sqrt{n}(\ln (n\CG(W)\big)$. 
This is true since 
$\CG(W)=\CR(B)$ according to Theorem~\ref{th:main}. 
\end{remark}

Theorem~\ref{th:BF} allows to break up the probabilistic study of $\CR$ 
into the study of the geometric condition~$\CG$ and the matrix condition $\kappa$. 
In particular, for random matrices~$A$ we have
$\Ex\log \CR(A) \leq \Ex\log\kappa(A) + \Ex\log \CG(A)$. 

In  the forthcoming paper~\cite{ambu:11} we will, based on methods from 
differential and spherical convex geometry, give tight bounds on 
the tail probability of $\CG$ and the expectation $\Ex\log \CG(A)$ for random $A\in\R^{m\times n}$ 
with independent standard Gaussian entries, for any regular cone~$C$. 

We will see next that the Grassmann condition $\CG(W)$ can be characterized 
in terms of angles. 
The {\em angle} $\sa(x,y)$ between two vectors $x,y\in\R^n\setminus \{0\}$ is defined by 
$\sa(x,y) := \arccos\big(\frac{x^Ty}{\|x\| \|y\|}\big)$, 
where $\|\;\|$ denotes the Euclidean norm. 
We define the angle between~$x$ and a subspace $W\in\GR$ by
$\sa(x,W) := \min\{\sa(x,y) \mid y \in W\setminus \{0\}\}$.
It is easy to see that 
$\sa(x,W) = \arccos\big(\|\Pi_W(x)\| /\|x\|\big)$, 
where $\Pi_W$ denotes the orthogonal projection onto~$W$.
Note that $\sa(x,W) \in [0,\pi/2]$.
If $W\cap C=\{0\}$, we define the angle between the cone~$C$ and 
the subspace~$W$ via
 $\sa(C,W) := \min\{ \sa(x,W) \mid x\in C\setminus \{0\}\}$.

\begin{proposition}\label{pro:GGsa}
We have
  \[ \CG(W)^{-1} = \begin{cases}
                      \sin\sa(C,W) \; , & \text{if } W\in\PGm \; ,
                   \\ \sin\sa(\breve{C},W^\bot) \; , & \text{if } W\in\DGm \; .
                   \end{cases} \]
\end{proposition}

\begin{remark}
In the dual feasible case,
the Grassmann condition as characterized in Proposition~\ref{pro:GGsa} 
was already considered by Belloni and Freund in~\cite[eq.~(3)]{BF:09}
and also the inequalities in Theorem~\ref{th:BF} 
were derived. What is missing in~\cite{BF:09} is the treatment of the primal feasible case,
and the geometric viewpoint in the Grassmann manifold, 
which leads to a completely transparent picture with regard to duality. 
\end{remark}

There is a natural Riemannian metric on  the Grassmann manifold $\GR$ that is invariant under the action of $O(n)$ and 
which is uniquely determined up to a scaling factor~\cite{helga}. 
The \emph{geodesic distance} $d_g(W_1,W_2)$ between $W_1,W_2\in\Gr_{n,m}$ is defined as the minimum 
length of a piecewise smooth curve in $\Gr_{n,m}$ connecting $W_1$ with $W_2$. 
One can nicely express $d_g(W_1,W_2)$ in terms of the principle angles between these subspaces, 
see~\eqref{eq:d_gd} below.

The following result states that measuring the distance of $W$ to $\SG_m(C)$ either with 
the projective distance or the geodesic distance leads to the same result.
The resulting differential-geometric characterization of the Grassmann condition is the key 
to the probabilistic analysis of $\CG(W)$ in~\cite{ambu:11}. 
We write $d_g(W,\SG_m) :=\min\{ d_g(W,W') \mid W' \in\SG_m\}$.
 
\begin{thm}\label{th:CGcharDG}
We have $d_p(W,\SG_m) = \sin d_g(W,\SG_m)$ for $W\in\Gr_{n,m}$. 
\end{thm}

The remainder of the paper is devoted to the proofs of the results
stated in this section.
The proof of Theorem~\ref{th:CGcharDG} will be given in Section~\ref{se:GRdist},
whereas the proofs of the remaining statements will be given in Section~\ref{sec:matrix-perturb}.
The paper closes with a comparison of the Grassmann condition with 
the GCC condition for the cone $C=\IR_+^n$.

\bigskip
\noindent{\bf Acknowledgments:} We are grateful to the criticism of the anonymous referees, which 
led to an improved presentation of the paper. 

\section{Preliminaries}\label{se:prelim}

We first recall the fundamental notions of principle angles between linear subspaces,
going back to Jordan~\cite{Jord}.  

Let $W_1,W_2\in\GR$ and let $B_i\in\R^{m\times n}$ be balanced such that $W_i=\im(B_i^T)$, 
i.e., the rows of $B_i$ form an orthonormal basis of $W_i$.
Let $\s_1\ge\ldots\ge\s_m$ denote the singular values of $B_1B_2^T$. 
Note that $\s_1 =\|B_1B_2^T\| \le \|B_1\|\,\|B_2^T\|=1$. 
The {\em principal angles} $\a_1,\ldots,\a_m$ between $W_1$ and $W_2$ 
are defined as $\a_i:=\arccos\s_i\in [0,\pi/2]$, cf.~\cite{GB:73}.

The principal angles depend only on the pair $W_1,W_2$ of subspaces. 
Their relevance derives from the known fact~\cite{Jord,Wo:67} 
that two pairs of subspaces in $\GR$ lie in the same $O(n)$-orbit iff they 
have the same (ordered) vector~$\a=(\a_1,\ldots,\a_m)$ of principal angles. 
It follows that any orthogonal invariant metric on $\GR$ is expressible in 
terms of the principle angles. In fact, the following is true:
\begin{equation}\label{eq:d_gd}
 d_p(W_1,W_2) =  \sin\|\a\|_\infty = \sin \max_i \alpha_i ,\quad 
 d_g(W_1,W_2) = \|\a\|_2 = \sqrt{\alpha_1^2+\ldots+\alpha_m^2} .
\end{equation}
See ~\cite[\S12.4.3]{GoLoan} or~\cite[\S5.3]{Stewart} for a proof of the 
first equality and~\cite{Wo:67} for a proof of the second one. 
It follows that 
$d_p(W_1,W_2) \le \sin d_g(W_1,W_2)$. 

The nonzero principal angles between 
$W_1$ and $W_2$ coincide with the nonzero principal angles between 
$W_1^\bot$ and $W_2^\bot$, cf.~for example~\cite[Thm.~3]{MB:92}.
Combined with \eqref{eq:d_gd}, this yields
\begin{equation}\label{eq:perp}
 d_p(W_1^\bot,W_2^\bot) = d_p(W_1,W_2), \quad
 d_g(W_1^\bot,W_2^\bot) = d_g(W_1,W_2) . 
\end{equation}

We shall need the polar decomposition of a matrix.
For any $A\in\R^{m\times n}_*$ we define $S=\sqrt{AA^T}$ and $B=S^{-1}A$.
Then $S$ is positive definite and 
$BB^T = S^{-1}AA^TS^{-1}=I$, hence $B$ is balanced. 
One calls $A=SB$ the {\em polar decomposition} of~$A$, cf.~\cite[\S4.2.10]{GoLoan}. 
It is clear that $A$ and $S$ have the same singular values~\cite[\S2.5.3]{GoLoan}
$\s_1\ge\cdots\ge\s_m >0$.
In particular, 
$\|A\| = \s_1 =\|S\|$ and $\|A^\dagger\| = \s_m^{-1 }=\|S^{-1}\|$.  
We shall call $B$ the \emph{balanced approximation} of~$A$. 
Replacing $A$ by its balanced approximation~$B$ may be interpreted as 
a preconditioning process. 

A linear map between Euclidean vector spaces 
is called {\em isometrical} iff it preserves the inner product 
between vectors.
The next lemma is well known and summarizes some of the defining properties 
of balanced matrices. We include a proof for lack of a suitable reference. 

\begin{lemma}\label{le:bal} 
For a matrix $A\in\R^{m\times n}_*$ the following conditions are equivalent:                                                                                                                     
\begin{enumerate}
\item[(1)] $A$ is balanced,
\item[(2)]  the map $A^T\colon\R^m\to\R^n$ is isometrical,
\item[(3)] $A$ restricted to $\ker(A)^\bot$ is isometrical,
\item[(4)] $\kappa(A) =\|A\| =1$.
\end{enumerate}
In this case, $A^TA$ equals the orthogonal projection onto 
$\im(A^T) =\ker(A)^\bot$. 
\end{lemma}

\begin{proof}
$A^T$ is isometrical iff 
$(A^T e_i)^T (A^T e_j) = \delta_{ij}$ for all $i,j$. 
This is equivalent to $e_i^TAA^Te_j =\delta_{ij}$, 
or $AA^T=I_m$, which means that $A$ is balanced. 
We have thus verified the equivalence of (1) and (2). 

We show now the equivalence of (1) and (3). 
Suppose that $A$ is balanced. Then we have 
$(A A^Ty_1)^T A A^Ty_2 = y_1^T y_2 = (A^Ty_1)^T (A^T y_2) $ 
for all $y_1,y_2\in\R^m$. 
Hence the map $A$ restricted to $\ker(A)^\bot=\im A^T$ 
is isometrical.

To see the converse, suppose that  
$(A A^Ty_1)^T A A^Ty_2 = (A^Ty_1)^T (A^T y_2)$.
This means that $AA^TAA^T= AA^T$. 
Since $AA^T$ is invertible we get $AA^T=I_m$. 

To see the equivalence of (1) and (4) suppose that 
$A$ is balanced. Then $\|A\|=1$ by (2) and 
$\|A^\dagger\|=1$ by (3). 
Conversely, assume that $\|A\| = \|A^\dagger\|=1$. 
Since $\|A\|$ is the largest and  $\|A^\dagger\|^{-1}$ is 
the smallest singular value of $A$, it follows 
that $A=U\,(I_m \; 0)\,V$ for orthogonal matrices $U\in O(m)$ and 
$V\in O(n)$, cf.~\cite{GoLoan}. Hence (2) and thus (1) is true. 

For the last assertion, let $x=x_1+x_2$ with $x_1\in\ker(A)$ and 
$x_2\in\im(A^T)$, say $x_2=A^Ty_2$. 
Then $A^TAx_1=0$ and $A^TAx_2= A^TAA^Ty_2=A^Ty_2=x_2$. 
Hence $A^TA$ equals the orthogonal projection onto $\im(A^T)$. 
\end{proof}

\begin{remark}\label{re:EY}
Recall that $\RD$ denotes the set of rank deficient matrices in  $\R^{m\times n}$.  
The Eckart-Young Theorem~\cite{EckYou}, see also~\cite[\S2.5.5]{GoLoan},  
states that $d(A,\RD) = \|A^\dagger\|^{-1}$,
where $A^\dagger$ denotes the Moore-Penrose pseudoinverse of $A$. 
This implies that $d(A,\RD)=1$ iff $A$ is balanced. 
\end{remark}

\section{Distances in the Grassmann manifold}\label{se:GRdist}

We show here that different choices of distances in the Grassmann manifold 
lead to the same notion of Grassmann condition. 

\begin{lemma}\label{le:dga}
Let $W_1,W_2\in\GR$ such that $\overline{W}:=W_1\cap W_2$ has the dimension~$m-1$
and let the line $L_i$ denote the orthogonal 
complement of $\overline{W}$ in $W_i$.
Then the principal angles $\alpha_1\leq \ldots\leq\alpha_m$ between $W_1$ and $W_2$
are given by $\alpha_1=\ldots=\alpha_{m-1}=0$ and
$\alpha_m=\sa(L_1,L_2)$. 
Moreover, $d_p(W_1,W_2) = \sin\a_m$ and $d_g(W_1,W_2) =\a_m$.
\end{lemma}

\begin{proof}
Let $x_i\in L_i$, $i=1,2$, such that $\|x_1\|=\|x_2\|=1$ and $\sa(x_1,x_2)=\sa(L_1,L_2)$.
If the rows of $B_i\in\IR^{m\times n}$ consist of $x_i$ and an orthonormal basis of $\overline{W}$, 
then $B_i$ is balanced and $\im(B_i^T)=W_i$. 
We have $B_1B_2^T=\begin{pmatrix} x_1^Tx_2 & 0 \\ 0 & I_{m-1}\end{pmatrix}$. 
Hence the vector of principal angles between $W_1$ and $W_2$ equals
$(0,\ldots,0,\a_m)$, where  $\a_m=\arccos|x_1^Tx_2|$. 
Equation~\eqref{eq:d_gd} implies $d_p(W_1,W_2)= \sin\a_m$ and 
$d_g(W_1,W_2)= \a_m$ as claimed. 
 \end{proof}

\begin{lemma}\label{le:Ex}
For $x\in\R^n\setminus \{0\}$ set 
$\Es_x:=\{ W'\in\Gr_{n,m} \mid x\in W'\}$. Then we have for $W\in\GR$
$$
 d_p(W,\Es_x) =  \sin d_g(W,\Es_x) = \sin\sa(x,W) \; .
$$ 
\end{lemma}

\begin{proof}
Let $w\in W\setminus0$ be such that $\theta:=\sa(x,W) =\sa(x,w)$.
Without loss of generality we may assume $\|x\|=\|w\|=1$.
We have the orthogonal decomposition 
$W = \overline{W} + \R w$, where $\overline{W} := W \cap w^\bot$. 
Note that 
$x\in\overline{W}^\bot$ since $x-\cos(\theta)\,w \in W^\bot$. 
Hence we have an orthogonal decomposition 
$W' := \overline{W} + \IR\,x$ and 
Lemma~\ref{le:dga} implies that 
$$
 d_p(W,\Es_x) \le \sin d_g(W,\Es_x) \le \sin d_g(W,W') = \sin\theta .
$$

It remains to prove that $d_p(W,\Es_x) \geq \sin\theta$,  
For this, take any space $W'\in\Es_x$ and 
put $\ol{W} := W'\cap x^\bot$. Then we have an 
orthogonal decomposition $W'= \ol{W} + \IR\, x$.
In order to calculate the principal angles between $W$ and $W'$, 
let $b_1,\ldots,b_{m-1}$ be an orthonormal basis of $\ol{W}$
and consider the balanced matrix 
$\tilde{B}\in\IR^{m\times n}$ with the rows $x,b_1,\ldots,b_{m-1}$.
Then we have 
$$
\tilde{B}w = (x^Tw, x^T b_1,\ldots,x^T b_{m-1}) = (x^Tw,0,\ldots,0)
$$ 
since $x \in \ol{W}^\bot$. Therefore, 
$\|\tilde{B}w\|= x^Tw = \cos\theta$.
Let $B\in\IR^{m\times n}$ denote a balanced matrix 
consisting of the first row $w$ and an orthonormal basis 
of $\overline{W}$.
Using the fact that the smallest singular value~$\sigma_m$
of the matrix $\tilde{B}B^T$ is given by
$\sigma_m = \min_{\|y\|=1} \| \tilde{B}B^Ty\|$
(cf.~\cite[Thm.~I.4.3]{Stewart}),
we conclude
$\sigma_m\leq \|\tilde{B}B^Te_1\|=\|\tilde{B}w\|=\cos\theta$.
If we denote by $\alpha=(\alpha_1,\ldots,\alpha_m)$ the
vector of principal angles between $W'$ and $W$, 
we get 
$\|\alpha\|_\infty \ge \arccos\sigma_m\ge \theta$ 
and hence, using~\eqref{eq:d_gd}, 
$d_g(W',W)  = \sin \|\alpha\|_\infty \ge \sin\theta$.
\end{proof}

\begin{proof}[Proof of Proposition~\ref{pro:GGsa}]
Note that by definition, 
$$
 \mbox{$\DGm = \{W' \in \GR \mid W'\cap C \ne \{0\} \} = \bigcup_{x\in C\setminus \{0\}} \Es_x$} \; .
$$
Combinig this with Lemma~\ref{le:Ex} we obtain
\begin{equation}\label{eq:sigh}
 d_p(W,\DGm ) =  \inf_{x\in C\setminus \{0\}} d_p(W,\Es_x) = \inf_{x\in C\setminus \{0\}} \sin\sa(x,W) = \sin\sa(C,W) . 
\end{equation}
Moreover, Proposition~\ref{pro:Gtop} implies $d_p(W,\SGm) = d_p(W,\DGm)$ 
in the case $W\in\PGm$. 
\end{proof}

\begin{proof}[Proof of Theorem~\ref{th:CGcharDG}]
Arguing as for~\eqref{eq:sigh} with $d_g$ instead of $d_p$ and using Lemma~\ref{le:Ex}, 
we get 
$$
  \sin d_g(W,\DG_m(C)) = \inf_{x\in C\setminus \{0\}} \sin d_g(W,\Es_x) 
     = \inf_{x\in C\setminus \{0\}} d_p(W,\Es_x) = d_p(W,\DG_m(C)) .
$$  
This implies the assertion in the case $W\in\PG_m(C)$. 

For the case $W\in\DG_m(C)$ we recall that the involution 
$\iota_m\colon\Gr_{n,m} \to \Gr_{n,n-m},\; W \mapsto W^\bot$
is an isometry both for $d_p$ and $d_g$, cf.~\eqref{eq:perp}.
This implies 
$$
 d_p(W^\perp,\Sigma_{n-m}(\breve{C})) = d_p(W,\Sigma_m(C)), \quad
 d_g(W^\perp,\Sigma_{n-m}(\breve{C})) = d_g(W,\Sigma_m(C)), 
$$
as a consequence of the duality relations~\eqref{eq:sym}. 
Hence the assertion in the case $W\in\DG_m(C)$ follows 
from the one for $W\in\PG_m(C)$.
\end{proof}

We complement these results by a different characterization of the projective distance.
The {\em Hausdorff distance} between $W_1$ and $W_2$ in $\GR$ is defined as 
\begin{equation}\label{eq:d_H(W_1,W_2)}
  d_H(W_1,W_2) = \max\big\{ \sphericalangle(x,W_2)\mid x\in W_1\setminus \{0\} \big\} \; .
\end{equation}
This notion of distance evolves from the identification
of a subspace $W\in\Gr_{n,m}$ with the
subsphere $W\cap S^{n-1}$ of the unit sphere.
As the unit sphere is a metric space,
also the set of closed subsets of $S^{n-1}$ is endowed with a natural
metric, which is known as the Hausdorff metric (cf.~for example~\cite[\S1.2]{M:06}).
For subspheres resp.~subspaces via the above identification,
this metric is given as stated in~\eqref{eq:d_H(W_1,W_2)}.

\begin{proposition}\label{lem:d_p,d_H}
For $W_1,W_2\in\Gr_{n,m}$ we have 
$d_p(W_1,W_2) =\sin d_H(W_1,W_2)$.
\end{proposition}

\begin{proof}
Let $\alpha$ denote the vector of principal angles between $W_1$ and $W_2$. 
Further, 
let the rows of $B_i\in\IR^{m\times n}$
form an orthonormal basis of $W_i$, for $i=1,2$.
Using the characterization of the smallest singular value of a matrix $A$ via
$\min_{\|y\|=1} \| Ay\|$, and using Lemma~\ref{le:bal}, we get
  \[ \cos \|\alpha\|_\infty = \min_{\|y\|=1} \| B_2B_1^Ty\| = \min_{x\in W_1\cap S^{n-1}} \|B_2 x\|
   = \min_{x\in W_1\cap S^{n-1}} \|\Pi_{W_2} x\| \; . \]
Applying the arccosine thus yields
\[ 
  \|\alpha\|_\infty 
      = \max_{x\in W_1\cap S^{n-1}} \arccos(\|\Pi_{W_2} x\|) = \max_{x\in W_1\cap S^{n-1}} \sa(x,W_2)
      = d_H(W_1,W_2) \; . 
\]
The claim follows with \eqref{eq:d_gd}.
\end{proof}

\section{Perturbations of balanced operators}\label{sec:matrix-perturb}

We provide here the proofs of the remaining results stated in the introduction.

Suppose $B\in\R^{m\times n}_*$ and consider a line $\R x$  that is not contained in $W=\im (B^T)$. 
What is the minimum norm of a perturbation $\Delta$ of $B$ such that 
$\R x\subseteq\im (B^T + \Delta^T)$? 
The lemma below shows that, for a balanced matrix~$B$, 
the answer is given by $\sin\sa(x,W)$. 
We also answer the analogous question with regard to $\ker(B)$. 

The {\em Frobenius norm} of a matrix $A\in\R^{m\times n}$ is defined as 
$\|A \|_F:=\tr(A A^T)^{1/2}$. Recall that $\|A\|$ denotes the spectral norm,
that is, the largest singular value. Both matrix norms are invariant 
under the left and right multiplication with orthogonal matrices. 
We will frequently use the well known fact that 
$\|x y^T\| =  \|x y^T\|_F = \|x\|\, \|y\|$
for $x\in\R^m$ and $y\in\R^n$. 
(This follows easily from the orthogonal invariance.)

\begin{lemma}\label{prop:balanced-favor-1}
Let $B\in\IR^{m\times n}$ be balanced and $W:=\im(B^T)$. Furthermore, let $x\in \IR^n\setminus \{0\}$ 
and put $\alpha:=\sphericalangle(x,W)$, $\beta:=\sphericalangle(x,W^\bot)=\frac{\pi}{2}-\alpha$. 
\begin{enumerate}

\item Then for all $\Delta,\Delta'\in\IR^{m\times n}$, we have 
\begin{align*}
   x\in \im(B^T+\Delta^T) & \;\;\Rightarrow\;\; \|\Delta\| \geq \sin\alpha \; ,
\\[1mm] x\in \ker(B+\Delta') & \;\;\Rightarrow\;\; \|\Delta'\| \geq \sin\beta \; . 
\end{align*}

\item Additionally suppose that $x\not\in W^\bot$. Then there 
exist matrices $\Delta,\Delta'\in\R^{m\times n}$ of rank at most one such that 
$\|\Delta\|_F = \sin\alpha$, $\|\Delta'\|_F = \sin\beta$, and
$x\in\im(B^T+\Delta^T)$, $x\in\ker(B+\Delta')$.  
\end{enumerate}
\end{lemma}

\begin{proof}
1. If $x\in\im(B^T+\Delta^T)$, then there exists $v\in S^{m-1}$ and $r>0$ such that $(B^T+\Delta^T) v=r x$. 
Then we have, as $\|B^Tv\|=1$,
\[ 
  \|\Delta\| \;\geq\; \|\Delta^T v\| = \|r  x - B^T v\| \;\geq\; \sin\sa(x,B^Tv) \;\geq\; \sin\alpha \; .
\]
If $(B+\Delta')\cdot x=0$, we have, writing $x^\circ := \|x\|^{-1} x$,
\[ 
   \|\Delta'\| \;\geq\; \|\Delta'  x^\circ\| \;=\; \|Bx^\circ\| \;=\; \|B^TB x^\circ\| \;=\; \cos\alpha \;=\; \sin\beta  \; , 
\]
as $B^TB$ is the orthogonal projection onto $W$, cf.\ Lemma~\ref{le:bal}. 

2. Without loss of generality, we may assume $\|x\|=1$. Then the matrices in the lemma may be chosen as 
\[ 
\Delta := Bp \, (\cos(\alpha)\,x-p)^T \;,\quad \Delta' := -B x x^T \; , 
\]
where 
$p := \cos(\alpha)^{-1}\, B^TBx$ is  
the normalized orthogonal projection of $x$ on $W$. These are matrices
of rank at most~$1$.
Using the fact $\| yz^T\|_F = \|y\| \|z\|$ we obtain for their 
Frobenius norms
\[
  \|\Delta\|_F = \|Bp\|\cdot \|\cos(\alpha)\,x-p\| = \sin\alpha 
\]
and 
\[
  \|\Delta'\|_F  = \|B x x^T\|_F = \|B x\|\cdot \|x\| =
\|B^TB x\|\cdot \|x\| = \cos\alpha = \sin\beta \; .
\]
Furthermore, we have
\begin{align*}
   (B^T+\Delta^T) Bx & = B^TBx +  (\cos(\alpha)\,x-p) \, p^TB^TBx
\\ & = \cos(\alpha)\,p + \cos(\alpha)\cdot (\cos(\alpha)\,x-p) \, p^Tp = \cos^2(\alpha)\cdot x \; ,
\end{align*}
which shows that $x\in\im(B^T+\Delta^T)$. 
Moreover, we have $(B+\Delta')x = Bx - B x x^T x = 0$, which shows that $x\in\ker(B+\Delta')$.
\end{proof}

\begin{proof}[Proof of Theorem~\ref{th:main}] 
Let $B\in\R^{m\times n}$ be balanced such that $W=\im(B^T)$.
We shall distinguish two cases.

(i)  We assume that $W\in\PGm$. 
Lemma~\ref{prop:balanced-favor-1}(1) implies that for any $x\ne 0$, 
\begin{equation}\label{eq:uno}
   \inf\{\|\Delta\| \mid x\in\im(B^T+\Delta^T)\} \geq  \sin\sa(x,W) \; .
\end{equation}
In the case $x\not\in W^\bot$, 
Lemma~\ref{prop:balanced-favor-1}(2) implies that equality holds 
(and the infimum is attained).
By a limit consideration it follows that equality also holds for $x\in W^\bot\setminus \{0\}$ 
(but the infimum may not be attained). 
It follows from the equality in~\eqref{eq:uno} that 
\begin{equation}\label{eq:hier}
 \inf\{\|\Delta\| \mid \im(B^T+\Delta^T) \cap C \ne 0\}  \ =\  
 \inf_{x\in C\setminus \{0\}} \sin\sa(x,W) \ =\ \sin\sa(C,W) .
\end{equation}
We conclude that 
$d(B,\Ds) = \sin\sa(C,W)$, where 
$\Ds:=\{A \in\R^{m\times n} \mid \im(A^T) \cap C\ne 0 \}$. 

By~\eqref{eq:star2}, we have 
$d(B,\FDR) = \min \{d(B,\Ds),d(B,\RD)\}$.
But the Eckart-Young Theorem 
(cf.\ Remark~\ref{re:EY}) implies 
$d(B,\RD) = 1$. 
Therefore we have 
by the definition of Renegar's condition number~\eqref{def:RCN}, 
$$
\CR(B)^{-1} = d(B,\SR) = d(B,\FDR) = d(B,\Ds) = \sin\sa(C,W) .
$$
Here we used the assumption $W\in\PGm$, which means $B\in\FPR$. 
Finally, Proposition~\ref{pro:GGsa} states that 
$\CG(W)^{-1} = \sin\sa(C,W)$. 
Hence we conclude that $\CG(W)= \CR(B)$ in the case $W\in\PGm$. 

(ii) We assume now that $W\in\DGm$, that is, $B\in\FDR$.
Lemma~\ref{prop:balanced-favor-1}(1) implies for $x\ne 0$ that 
$$
 \inf\{\|\Delta'\| \mid x\in\ker(B+\Delta')\} \;\ge\;\sin\sa(x,W^\bot) \; 
$$
and Lemma~\ref{prop:balanced-favor-1}(2) shows that equality holds. 
Taking the infimum over all nonzero $x\in\breve{C}$ it follows that 
\[ 
 \CR(B)^{-1} = d(B,\FPR) 
  = \inf\big\{ \|\Delta'\|\mid \ker\left( B + \Delta'\right)\cap \breve{C}\neq \{0\} \big\} 
  = \sin\sa(\breve{C},W^\bot) \; . 
\]
On the other hand, 
$W\in\DGm$ implies $W^\bot\in\PG_m(\breve{C})$ and therefore, 
Proposition~\ref{pro:GGsa} yields
$\CG_{\breve{C}}(W^\bot)^{-1} = \sin\sa(\breve{C},W^\bot)$.
Hence we conclude that 
$\CG_{\breve{C}}(W^\bot) = \CR_C(B)$. 
Finally, due to~\eqref{eq:SGdual}, we get 
$\CG_C(W) = \CG_{\breve{C}}(W^\bot)  = \CR_C(B)$, 
which completes the proof of
Theorem~\ref{th:main}. 
\end{proof}

\begin{remark}\label{re:Frobnorm}
The proof of Theorem~\ref{th:main} shows that 
$d_F(A,\SR)= d(A,\SR)$ for a balanced matrix $A\in\R^{m\times n}$, 
where $d$ and $d_F$ denote the distances measured in the 
spectral and Frobenius norm, respectively. 
In fact, this equality also holds for the nonbalanced case. In the dual
feasible case $A\in\FDR$,
a perturbation $\Delta$ such that $A+\Delta\in\FPR$
and $\|\Delta\|_F=d(A,\SR)$ is given by
$\Delta=-App^T$, where $p\in \breve{C}\cap S^{n-1}$ is chosen such that
$\|Ap\|=\min\{\|Aq\|\mid q\in\breve{C}\cap S^{n-1}\}$, cf.~\cite[Lem.~3.2]{BF:09}.
In the primal feasible case $A\in\FPR$, 
the fact that one can find
rank-one perturbations $\Delta$ such that $A+\Delta\in\SR$ and
$\|\Delta\|_F=d(A,\SR)$ follows from~\cite[Prop.~3.5]{Pe:00}.
\end{remark}

\begin{proof}[Proof of Theorem~\ref{th:BF}] 
Let $A=SB$ be the polar decomposition of $A\in\R^{m\times n}_*$.
Then $B$ is balanced and $S=\sqrt{AA^T}$. 
Since $A$ and $S$ have the same singular values 
we have 
$\|S\|= \|A\|$ and $\|S^{-1}\|= \|A^\dagger\|$ 
(compare Section~\ref{se:prelim}).  
By the main Theorem~\ref{th:main} 
we have $\CG(W)=\CR(B)$. 
Thus, by the definition~\eqref{def:RCN} of Renegar's condition number, 
the assertion of Theorem~\ref{th:BF}
is equivalent to 
$$
 \frac1{d(B,\SR)} \;\leq\; \frac{\|A\|}{d(A,\SR)} \;\leq\; \|A\|\,\|A^\dagger\| \, \frac1{d(B,\SR)} ,
$$
or, equivalently,  
\begin{equation}\label{eq:dist-inequ}
\|S^{-1}\| \, d(B,\SR) \; \le\; d(A,\SR)\;\le\; \|S\|\, d(B,\SR)\; .
\end{equation}

To show the right-hand inequality, let $\tilde{B}\in\SR$ be 
such that $d(B,\SR)= \|B-\tilde{B}\|$. 
We define $\tilde{A} := S\tilde{B}$. 
Then we have $\tilde{A}\in\SR$ by the invariance of $\SR$ under the 
$\GL(m)$-left action on $\R^{m\times n}$. 
Therefore, 
$$
 d(A,\SR) \;\le\; \|A-\tilde{A}\| \;\le\; \|S(B-\tilde{B})\| \;\le\; 
 \|S\| \, \|B-\tilde{B}\| = \|S\|\, d(B,\SR) . 
$$
For the left-hand inequality, let $\tilde{A}\in\SR$ be 
such that $d(A,\SR)= \|A-\tilde{A}\|$. 
We define $\tilde{B} := S^{-1}\tilde{A}$ and note that $\tilde{B}\in\SR$.  
Then we have 
  \[
 d(B,\SR) \;\le\; \|B-\tilde{B}\| \;\le\; \|S^{-1}(A-\tilde{A})\| \;\le\; 
 \|S^{-1}\| \, \|A-\tilde{A}\| = \|S^{-1}\|\, d(A,\SR) . \qedhere 
  \]
\end{proof}

Theorem~\ref{th:main} combined with a known characterization of 
Renegar's condition number in the primal feasible case 
implies the following result.
Let $B_n$ denote the closed unit ball in~$\IR^n$. 

\begin{cor}\label{th:CGr}
For $W\in \PGm$ we have 
$\CG(W)^{-1} = \max\{r\mid r\cdot B_n\cap W\subseteq \Pi_{W}(\breve{C}\cap B_n)\}$.
\end{cor}

\begin{proof}
We use the following known characterization of Renegar's condition number in the primal feasible case: 
For $A\in\FPR$ and $\|A\|=1$ we have (see~\cite{rene:95a} or~\cite[Cor.~3.6]{Pe:00})
\begin{equation*}
  \CR(A)^{-1} = \max\{r\mid r\cdot B_m\subseteq A (B_n\cap \breve{C})\} \; .
\end{equation*}
Let $A\in\FPR$ be balanced such that $W=\im(A^T)$. 
From Theorem~\ref{th:main} and from the above characterization of $\CR(A)^{-1}$ we get
$$
 \CG(W)^{-1} = 
     \max\{r\mid r\cdot B_m\subseteq A(B_n\cap \breve{C})\}
   = \max\big\{r\mid r\cdot A^T(B_m)\subseteq A^T A(B_n\cap \breve{C})\big\} \; ,
$$
where the last equality follows from the fact that the map $A^T\colon\IR^m\to\IR^n$ 
is isometrical, cf.~Lemma~\ref{le:bal}. 
Since $A^T A$ is the orthogonal projection~$\Pi_W\colon\IR^n\to W$, we get
\[ 
\CG(W)^{-1} = \max\big\{r\mid r\cdot B_n\cap W\subseteq \Pi_W(B_n\cap \breve{C})\big\} \; . \qedhere
\]
\end{proof}

Finally, for the sake of completeness, we provide a proof of the topological 
result stated in the introduction

\begin{proof}[Proof of Proposition~\ref{pro:Gtop}]
(1)  Note that 
$\Es:=\{ (x,W) \in S^{n-1}\times\Gr_{n,m} \mid x \in W\}$
is a closed subset of $ S^{n-1}\times\Gr_{n,m}$.  
The set $\DG_m(C)$ is obtained as the projection of 
the compact set $(C\times\Gr_{n,m})\cap\Es$ 
onto the second component.
A  standard compactness argument shows that $\DG_m(C)$ is closed. 
The closedness of $\PG_m(C) = \iota_m^{-1}(\DG_{n-m}(\breve{C}))$, cf.~\eqref{eq:sym}, 
follows from the closedness of $\DG_{n-m}(\breve{C})$ 
and the continuity of $\iota_m$.

(2) As $\PG_m(C)\cup \DG_m(C)=\GR$, we have $\GR\setminus\PG_m(C)\subseteq\DG_m(C)$,
and thus $\partial \PG_m(C)\subseteq\SG_m(C)$. Analogously, we have $\partial \DG_m(C)\subseteq\SG_m(C)$.

For the other inclusion suppose $W\in\SG_m(C)$, say $x\in W\cap C$ for some $x\ne 0$. 
There exists a sequence $x_k\in\inter(C)$ such that 
$\|x-x_k\|\leq \frac{1}{k}$ for all $k>0$. 
Put $\overline{W}:=W\cap x^\bot$ and 
define $W_k:=\IR\,x_k+\overline{W}$.
Then $W_k$ converges to $W$ for $k\to\infty$.
As $x_k\in W_k\cap \inter(C)$ we have $W_k\in \GR\setminus\PG_m(C)$ by~\eqref{eq:alt} 
and hence $W\in\ol{\GR\setminus\PG_m(C)}$.
This proves $\SG_m(C)\subseteq\partial\PG_m(C)$.
We have thus shown that $\partial\PG_m(C) = \SG_m(C)$. 
The assertion $\partial\DG_m(C) = \SG_m(C)$ follows now 
from the duality relations~\eqref{eq:sym}.
\end{proof}

\section{Comparison with the GCC condition number}

The GCC condition number, introduced in~\cite{goff:80,ChC:01}, is only defined for the cone $C=\IR_+^n$. 
We briefly compare this notion with the Grassmann condition number in this special setting. 

For $p\in S^{m-1}$ and $0\leq\alpha\leq \pi$ we denote by 
$\capp(p,\alpha) := \{ x\in S^{m-1}\mid \langle x,p\rangle \geq \cos\alpha \}$
the spherical cap with center~$p$ and angular radius~$\alpha$. 
Let $A\in\IR^{m\times n}$ with nonzero columns $a_1,\ldots,a_n\in\IR^m$.
The GCC-condition number~$\CGCC(A)$ can be characterized as
$\CGCC(A)^{-1} = |\cos\rho|$, where $\rho$ is the minimum angle of a 
spherical cap containing all the points $a_i/\|a_i\|$, 
see~\cite{ChC:01} or \cite[\S 6.5]{COND}. 

The following relationship  between Renegar's condition number $\CR(A)$ and the GCC condition $\CGCC(A)$ 
was established in~\cite[Prop.~4-5]{ChC:01}:
$$
  \frac{\min_i\|a_i\|}{\|A\|}\cdot \CR(A) \;\leq\; \CGCC(A) \;\leq\; \sqrt{n}\cdot \CR(A) \; . 
$$
By Theorem~\ref{th:BF}, this immediately implies, setting  $W:=\im(A^T)$,
$$
  \frac{\min_i\|a_i\|}{\|A\|}\cdot \CG(W) \;\leq\; \CGCC(A) \;\leq\; \sqrt{n}\cdot \kappa(A)\cdot \CG(W) \; .
$$

The next two examples show that this estimate cannot be substantially improved in the sense that both 
quotients $\CG(W)/\CGCC(A)$ and $\CGCC(A)/\CG(W)$ cannot  be bounded as a function of the 
dimensions $m,n$ only. (The first example is from \cite[Ex. 3.1/4.1]{BF:09}.) 
Let $\veps>0$ and consider the matrices
\[ 
  A_\veps := \begin{pmatrix} 2\veps & 1 & 1 \\ 0 & -1 & 1 \end{pmatrix} \;,\qquad
  \tilde{A}_\veps := \begin{pmatrix} 1+\veps & 1+\veps & -1+\veps \\ -1 & -1 & 1 \end{pmatrix} \; . 
\]
We have $\CGCC(A_\veps) = \sqrt{2}$ since 
the smallest enclosing cap of the three points 
obtained by normalizing the columns of $A_\veps$ 
has the center $p=(1,0)$ and the angular radius $\rho=\pi/4$. Similarly, we obtain 
$\CGCC(\tilde{A}_\veps) = \frac{\sqrt{2}\sqrt{(1+\veps)^2 +1}}{\langle (1,1), (1+\veps,-1)\rangle}
 = \frac{2}{\veps}(1+o(1))$ for $\veps\to 0$. 
As for the Grassmann condition numbers, we put $W_\veps := \im(A_\veps^T)$
and define $\tilde{W}_\veps$ similarly. 
For $A_\veps$ we are in the dual feasible situation and 
$W_\veps^\perp$ equals the line spanned by $(1,-\veps,-\veps)$. 
Proposition~\ref{pro:GGsa} now easily implies that 
$\CG(W_\veps)^{-1} = \sin\alpha$, where 
$\cos\alpha= \frac{\langle (1,-\veps,-\veps), (1,0,0) \rangle}{\sqrt{1+2\veps^2}}$.
Hence, for $\veps\to 0$.  
$$ 
  \CG(W_\veps) = \frac1{\sin\alpha} = 
 \frac{\sqrt{1+2\veps^2}}{\veps\sqrt{2}} = \frac1{\veps\sqrt{2}}(1+o(1)) .
$$
Furthermore, $\tilde{W}_\veps$ equals the span of 
$(1,1,-1),(1,1,1)$ and hence is independent of $\veps$. 
Altogether, it follows that 
$\CG(W_\veps)/\CGCC(A_\veps)$ 
and 
$\CGCC(\tilde{A}_\veps)/\CG(\tilde{W}_\veps)$ 
are unbounded, as $\veps\to 0$.

\end{document}